\documentclass[11pt]{aart}

\usepackage[letterpaper, hmargin=1in, top=1in, bottom=1.2in, footskip=0.6in]{geometry}

\usepackage{titlesec}
\titleformat{\section}[block]{\filcenter\normalfont\bfseries\large}{\thesection.}{.5em}{}\titlespacing*{\section}{0pt}{2\baselineskip}{1\baselineskip}
\titleformat{\subsection}[runin]{\normalfont\bfseries}{\thesubsection.}{.4em}{}[.]\titlespacing{\subsection}{0pt}{2ex plus .1ex minus .2ex}{.8em}
\titleformat{\subsubsection}[runin]{\normalfont\itshape}{\thesubsubsection.}{.3em}{}[.]\titlespacing{\subsubsection}{0pt}{1ex plus .1ex minus .2ex}{.5em}
\titleformat{\paragraph}[runin]{\normalfont\itshape}{\theparagraph.}{.3em}{}[.]\titlespacing{\paragraph}{0pt}{1ex plus .1ex minus .2ex}{.5em}

\usepackage[labelfont=bf,font=small,labelsep=period]{caption}
\usepackage[T1]{fontenc}
\usepackage{lmodern}

\usepackage{amsmath}
\usepackage{amssymb}
\usepackage{amsfonts}
\usepackage{latexsym}
\usepackage{amsthm}
\usepackage{amsxtra}
\usepackage{amscd}
\usepackage{bbm}
\usepackage{mathrsfs}
\usepackage{bm}
\usepackage{mathtools}

\usepackage{tikz}
\usetikzlibrary{arrows,decorations.pathmorphing,backgrounds,positioning,fit,petri}

\usepackage{graphicx, color}

\definecolor{darkred}{rgb}{0.9,0,0.3}
\definecolor{darkblue}{rgb}{0,0.3,0.9}

\definecolor{vdarkred}{rgb}{0.6,0,0.2}
\definecolor{vdarkblue}{rgb}{0,0.2,0.6}
\usepackage[pdftex, colorlinks, linkcolor=vdarkblue,citecolor=vdarkred]{hyperref}

\usepackage{booktabs}
\usepackage[nottoc,notlof,notlot]{tocbibind}
\usepackage{cite} 

\numberwithin{equation}{section}
\numberwithin{figure}{section}

\theoremstyle{plain} 
\newtheorem{theorem}{Theorem}[section]
\newtheorem*{theorem*}{Theorem}
\newtheorem{lemma}[theorem]{Lemma}
\newtheorem*{lemma*}{Lemma}

\newtheorem*{corollary*}{Corollary}
\newtheorem{proposition}[theorem]{Proposition}
\newtheorem*{proposition*}{Proposition}

\newtheorem*{conjecture*}{Conjecture}

\theoremstyle{definition} 
\newtheorem{definition}[theorem]{Definition}
\newtheorem*{definition*}{Definition}

\newtheorem*{example*}{Example}
\newtheorem{remark}[theorem]{Remark}
\newtheorem*{remark*}{Remark}

\newtheorem*{assumption*}{Assumption}

\newtheorem{Th}{Theorem}[section]

\newtheorem{lem}[Th]{Lemma}

\theoremstyle{definition}

\long\def\symbolfootnote[#1]#2{\begingroup
\def\thefootnote{\fnsymbol{footnote}}\footnote[#1]{#2}\endgroup}

\newcommand{\f}[1]{\boldsymbol{\mathrm{#1}}} 
\newcommand{\bb}{\mathbb} 
\renewcommand{\cal}{\mathcal} 
 
\newcommand{\fra}{\mathfrak} 

\newcommand{\col}{\vcentcolon}

\renewcommand{\P}{\mathbb{P}}
\newcommand{\E}{\mathbb{E}}
\newcommand{\R}{\mathbb{R}}
\newcommand{\C}{\mathbb{C}}
\newcommand{\N}{\mathbb{N}}

\newcommand{\ee}{\mathrm{e}}
\newcommand{\ii}{\mathrm{i}}
\newcommand{\dd}{\mathrm{d}}
\newcommand*{\deq}{\mathrel{\vcenter{\baselineskip0.65ex \lineskiplimit0pt \hbox{.}\hbox{.}}}=}
\newcommand*{\eqd}{=\mathrel{\vcenter{\baselineskip0.65ex \lineskiplimit0pt \hbox{.}\hbox{.}}}}

\renewcommand{\leq}{\leqslant}
\renewcommand{\geq}{\geqslant}

\renewcommand{\epsilon}{\varepsilon}

\DeclareMathOperator{\val}{Val}

\newcommand{\la}{\label}

\newcommand{\bgt}{\begin{itemize}}
\newcommand{\ent}{\end{itemize}}

\newcommand{\p}{\mathbb{P}}

\newcommand{\blem}{\begin{lem}}
\newcommand{\elem}{\end{lem}}

\newcommand{\ff}{\frac{1}}

\newcommand{\bbm}{\begin{bmatrix}}
\newcommand{\ebm}{\end{bmatrix}}
\newcommand{\bes}{\begin{equation*}}
\newcommand{\be}{\begin{equation}}
\newcommand{\ees}{\end{equation*}}
\newcommand{\beqy}{\begin{eqnarray}}
\newcommand{\eeqy}{\end{eqnarray}}
\newcommand{\beq}{\begin{eqnarray*}}
\newcommand{\eeq}{\end{eqnarray*}}

\newcommand{\bpm}{\begin{pmatrix}}
\newcommand{\epm}{\end{pmatrix}}

\newcommand{\ind}[1]{\f 1 (#1)}

\newcommand{\indB}[1]{\f 1 \pB{#1}}

\newcommand{\pb}[1]{\bigl(#1\bigr)}
\newcommand{\pB}[1]{\Bigl(#1\Bigr)}
\newcommand{\pbb}[1]{\biggl(#1\biggr)}
\newcommand{\pBB}[1]{\Biggl(#1\Biggr)}

\newcommand{\qB}[1]{\Bigl[#1\Bigr]}
\newcommand{\qbb}[1]{\biggl[#1\biggr]}
\newcommand{\qBB}[1]{\Biggl[#1\Biggr]}

\newcommand{\h}[1]{\{#1\}}

\newcommand{\abs}[1]{\lvert #1 \rvert}

\newcommand{\norm}[1]{\lVert #1 \rVert}

\newcommand{\normbb}[1]{\biggl\lVert #1 \biggr\rVert}
\newcommand{\normBB}[1]{\Biggl\lVert #1 \Biggr\rVert}

\DeclareMathOperator{\tr}{Tr}

\DeclareMathOperator{\im}{Im}

\title{Local law and complete eigenvector delocalization for supercritical Erd{\H o}s--R\'enyi graphs}

\author{Yukun He\footnote{University of Zurich, Institute of Mathematics, {\tt yukun.he@math.uzh.ch}.} \and Antti Knowles\footnote{University of Geneva, Section of Mathematics, {\tt antti.knowles@unige.ch}} \and Matteo Marcozzi\footnote{University of Geneva, Section of Mathematics, {\tt matteo.marcozzi@unige.ch}.}}

\begin{document}
\maketitle

\begin{abstract}
We prove a local law for the adjacency matrix of the Erd\H{o}s-R\'enyi graph $G(N, p)$ in the supercritical regime $ pN \geq C\log N$ where $G(N,p)$ has with high probability no isolated vertices. In the same regime, we also prove the complete delocalization of the eigenvectors. Both results are false in the complementary subcritical regime. Our result improves the corresponding results from \cite{EKYY1} by extending them all the way down to the critical scale $pN = O(\log N)$.
 
A key ingredient of our proof is a new family of multilinear large deviation estimates for sparse random vectors, which carefully balance mixed $\ell^2$ and $\ell^\infty$ norms of the coefficients with combinatorial factors, allowing us to prove strong enough concentration down to the critical scale $pN = O(\log N)$. These estimates are of independent interest and we expect them to be more generally useful in the analysis of very sparse random matrices.
\end{abstract}


\section{Introduction}
Let $\cal A \in \{0,1\}^{N \times N}$ be the adjacency matrix of the Erd\H{o}s-R\'enyi random graph $G(N,p)$, where $p \equiv p_N \in (0,1)$. That is, $\cal A$ is real symmetric, and its upper-triangular entries are independent Bernoulli random variables with mean $p$. The Erd\H{o}s-R\'enyi graph exhibits a phase transition in its connectivity around the critical expected degree $pN = \log N$. Indeed, for fixed $\epsilon > 0$, if $pN \geq (1 + \epsilon) \log N$ then $G(N,p)$ is with high probability connected, and if $pN \leq (1 - \epsilon) \log N$ then $G(N,p)$ has with high probability isolated vertices. See e.g.\ \cite[Chapter 5]{vanH} for a clear treatment. The aim of this article is to investigate the spectral and eigenvector properties of $G(N,p)$ in the supercritical regime, where $G(N,p)$ is connected with high probability.

Our main result is a \emph{local law} for the adjacency matrix $\cal A$ in the supercritical regime $C \log N \leq pN \ll N$, where $C$ is some universal constant. In order to describe it, it is convenient to introduce the rescaled adjacency matrix
\begin{equation} \label{def:A}
A \deq \sqrt{\frac{1}{p(1-p)N}} \, \cal A
\end{equation}
so that the typical eigenvalue spacing of $A$ is of order $N^{-1}$. A local law provides control of the matrix entries $G_{ij}(z)$ of the Green function
\begin{equation}  \label{def_G}
G(z)\deq(A-z)^{-1}\,,
\end{equation}
where $z=E+\ii \eta$ is a spectral parameter with positive imaginary part $\eta \gg N^{-1}$ defining the spectral scale. Our main result states that that the individual entries $G_{ij}(z)$ of the Green function concentrate all the way down to the critical scale $pN = C \log N$: the quantity
\begin{equation*}
\max_{i,j} \abs{G_{ij}(z) - \delta_{ij} m(z)}
\end{equation*}
is small with high probability for all $\eta \gg N^{-1}$; here $m(z)$ is the Stieltjes transform of the semicircle law.

Such local laws have become a cornerstone of random matrix theory, ever since the seminal work \cite{ESY1,ESY2} on Wigner matrices. They serve as fundamental tools in the study of the distribution of eigenvalues and eigenvectors, as well as in establishing universality in random matrix theory.

A local law has two well-known easy consequences, one for the eigenvectors and the other for the eigenvalues of $A$.
\begin{enumerate}
\item
The first consequence is \emph{complete eigenvector delocalization}, which states that the normalized eigenvectors $\{\f u_i\}$ of $A$ satisfiy with high probability
\begin{equation} \label{deloc_intro}
\max\{\|\f u_i\|_{\infty} \col 1\leq i \leq N\} \lesssim N^{-1/2}
\end{equation}
where $\lesssim$ denotes a bound up to some factor $N^{o(1)}$.
\item
The second consequence is a local law for the density of states, which states that the Stieltjes transform of the empirical eigenvalue distribution
\begin{equation*}
s(z) \deq \frac{1}{N} \tr G(z) = \frac{1}{N} \sum_{i = 1}^N \frac{1}{\lambda_i(A) - z}
\end{equation*}
satisfies $\abs{s(z) - m(z)} = o(1)$ with high probability for all $\eta \gg N^{-1}$. Informally, this means that the semicircle law holds down to very small spectral scales.
\end{enumerate}

It is a standard exercise to show that if $pN \to \infty$ then the (global) semicircle law for $s(z)$ holds, stating that $s(z) \to m(z)$ with high probability for any fixed $z \notin \R$.
On the other hand, it is not hard to see that both consequences (i) and (ii) are wrong in the subcritical regime $pN \leq (1 - \epsilon) \log N$. Indeed, in the subcritical regime there is with high probability an isolated vertex, with an associated eigenvector localized at that vertex. Hence, the left-hand side of \eqref{deloc_intro} is equal to one and (i) fails. To see how (ii) fails in the subcritical regime, let us set $pN = \kappa \log N$ and denote by $Y$ the number of isolated vertices. From \cite[Proposition 5.9]{vanH} we find that
\begin{equation*}
\E Y = N^{1 - \kappa} \pbb{1 + O \pbb{\frac{\kappa^2 (\log N)^2}{N}}}\,,
\end{equation*}
from which we deduce that $\P(Y \geq N^{1 - \kappa} / 4) \geq N^{-\kappa} / 4$ (since $Y \leq N$). Thus we find that with probability at least $N^{-\kappa} / 4$, the adjacency matrix $A$ has at least $N^{1 - \kappa}/4$ zero eigenvalues. We conclude that with probability at least $N^{-\kappa} / 4$ we have
\begin{equation*}
\im s(z) \geq \frac{\eta}{4 N^\kappa \abs{z}^2}\,.
\end{equation*}
For $\kappa < 1$, setting $z = \ii N^{- \frac{1 + \kappa}{2}}$ yields a contradiction to the estimate (ii), since $\abs{m(z)} \leq 1$. (The precise meaning of ``high probability'' in (ii) is sufficiently strong to rule out events of probability $N^{-\kappa} / 4$; see \eqref{semi} below).

Thus, our assumption $pN \geq C \log N$ is optimal up to the value of the numerical constant $C$. A local law for the Erd\H{o}s-R\'enyi graph was previously proved in \cite{EKYY1, EKYY2} under the assumption $pN \geq (\log N)^6$, and the contribution of this paper is therefore to cover the very sparse range $C \log N \leq pN \leq (\log N)^6$.

Next, we say a few words about the proof. Our strategy is based on the approach introduced in \cite{ESY1,ESY2,EYY1} and subsequently developed for sparse matrices in \cite{EKYY1,EKYY2}. Thus, we derive a self-consistent equation for the Green function $G$ using Schur's complement formula and large deviation estimates, which is then bootstrapped in the spectral scale $\eta$ to reach the smallest scale $N^{-1}$. The key difficulty in proving local laws for sparse matrices is that the entries are sparse random variables, and hence fluctuate much more strongly than in the Wigner case $p \asymp 1$. To that end, new large deviation estimates for sparse random vectors were developed in \cite{EKYY1}, which were however ineffective below the scale $(\log N)^6$.

The key novelty of our approach is a new family of multilinear large deviation bounds for sparse random vectors. They are optimal for very sparse vectors, and in particular allow us to reach the critical scale $pN = C \log N$. They provide bounds on multilinear functions of sparse vectors in terms of mixed $\ell^2$ and $\ell^\infty$ norms of their coefficients. We expect them to be more generally useful in a variety of problems on sparse random graphs. To illustrate them and how they are applied, consider a sparse random vector $X \in \R^N$ which is a single row of the matrix $A - \E A$. Let $(a_{ij})$ be a symmetric deterministic matrix. Then, for example, we have the $L^r$ bound
\begin{equation} \label{lde_intro}
\normbb{\sum_{i \neq j} a_{ij} X_i X_j}_r \leq \pbb{\frac{4 r}{1 + (\log (\psi / \gamma))_+} \vee 4}^2 (\gamma \vee \psi)\,,
\end{equation}
where
\begin{equation*}
\gamma \deq \pbb{\max_i \frac{1}{N} \sum_{j} \abs{a_{ij}}^2}^{1/2}\,, \qquad  \psi \deq \frac{\max_{i,j} \abs{a_{ij}}}{pN}\,.
\end{equation*}
We first remark that we have to take $r$ to be at least $\log N$. Indeed, our proof consists of an order $N^{O(1)}$ uses of such large deviation bounds. To compensate the factor $N^{O(1)}$ arising from the union bound, we therefore require bounds smaller than $N^{-D}$ on the error probabilities for any fixed $D>0$, which we obtain (by Chebyshev's inequality) from the large deviation bounds for $r = \log N$. The crucial feature of the bound \eqref{lde_intro} is the logarithmic factor in the denominator. Without it, there is nothing to compensate the factor $r^2 \geq (\log N)^2$ in the numerator, as $\psi \asymp (pN)^{-1} \asymp (\log N)^{-1}$ in the critical regime. Thus, the applicability of \eqref{lde_intro} hinges on the fact that $\psi / \gamma$ is sufficiently large; this assumption can in fact be verified in all of our applications of \eqref{lde_intro}.

Armed with these large deviation estimates, we can follow the basic approach of \cite{EYY1,EKYY1} to conclude the local law by a bootstrapping; the main difference is that we have to be cautious to work as much as possible with $L^r$ norms instead of the more commonly used high-probability bounds.

We remark that, in the recent years, a new approach to proving local laws has emerged \cite{LS1,HKR,HK}, which replaces the row by row operations of the Schur complement formula with operations on individual entries performed by the cumulant expansion (or generalized Stein's formula). This new approach is considerably more versatile and general than the one based on Schur's complement formula. However, it encounters serious difficulties with very sparse matrices, owing to the fact that it requires the computation of high moments of the Green function entries. For the very sparse scales that are of interest to us, these moments may be large, although the entries themselves are small with high probability. The underlying phenomenon is that the exceptional events on which the Green function is large do not have small enough probability to ensure the good behaviour of high moments.

Our proof is short and self-contained. For conciseness and clarity, we focus on the simple model of the Erd\H{o}s-R\'enyi graph, and do not aim for optimal error bounds. However, a straightforward modification of our method, combined with the deterministic analysis of the quadratic vector equation developed in \cite{AEK1,AEK2}, allows to extend our results to more general random graphs with independently chosen edges and general variance profiles, such as the stochastic block model. We do not pursue this direction further in this paper.

We conclude this section with a summary of some related results. The bulk universality for Erd{\H o}s--R\'enyi graphs was first proved in \cite{EKYY2} for $p \gg N^{-1/3}$, and it was later pushed to $p \gg N^{\varepsilon-1}$ for any fixed $\varepsilon>0$ in \cite{HL}. The edge universality was first proved for $p \gg N^{-1/3}$ in \cite{EKYY2}, and later extended to $p \gg N^{-2/3}$ in \cite{LS1}. More recently in \cite{HLY}, it was proved that the extreme eigenvalues has Gaussian fluctuations for $N^{-7/9}\ll p\ll N^{-2/3}$. We emphasize that the local law proved in this paper is, on its own, not precise enough to generalize the above universality results to the entire supercritical regime.

The extreme eigenvalues of $A$ in the supercritical and subcritical regimes were recently investigated in \cite{BBK1,BBK2}, where the authors proved that in the supercritical regime the extreme eigenvalues converge to the spectral edges $\pm 2$, and in the subcritical regime a fraction of eigenvalues leave the bulk $[-2,2]$ to become a cloud of outlier eigenvalues.

Local laws in the sparse regime have also been established for the random regular graph $G_{N,d}$, defined as the uniform probability distribution over all graphs on $N$ vertices such that every vertex has degree $d$. Because the degree of any vertex is fixed to be $d$, the random regular graph $G_{N,d}$ is much more stable in the sparse regime than the Erd\H{o}s-R\'enyi graph $G(N,d/N)$ with the same expected degree. In particular, $G_{N,d}$ does not exhibit a connectivity crossover and remains connected down to $d = 3$. A local law for the random regular graph $G_{N,d}$ for $d \geq (\log N)^4$ was proved in \cite{BKY15} and for fixed but large $d$ in \cite{BHY1}.

We organize the paper as follows. In Section \ref{sec:result} we state our main results. In Section \ref{sec:LDE} we prove a new family of multilinear large deviation estimates for sparse random vectors. In Section \ref{sec:proof} we use the results in Section \ref{sec:LDE} to complete the proof.

\section{Results} \label{sec:result}
For convenience, in the remaining of this paper we introduce the new variable 
$$
q \deq \sqrt{pN} \in (0,N^{1/2})\,.
$$
We consider random matrices of the following class; it is an easy exercise to check that $A$ defined in \eqref{def:A} in terms of $G(N,p)$ satisfies the following conditions.

\begin{definition} [Sparse matrix] \label{def:sperse} Let $q \in (0,N^{1/2})$. A sparse matrix is a real symmetric $N\times N$ matrix $H=H^* \in \bb R^{N \times N}$ whose entries $H_{ij}$ satisfy the following conditions.
	\begin{enumerate}
		\item[(i)] The upper-triangular entries ($H_{ij}\col 1 \leq i \leq j\leq N$) are independent.
		\item[(ii)] We have $\bb E H_{ij}=0$ and $ \bb E H_{ij}^2=(1 + O(\delta_{ij}))/N$ for all $i,j$.
		\item[(iii)] For any $k\geq 3$, we have
		$\bb E|H_{ij}|^k \leq 1/(Nq^{k-2})$
		for all $i,j$. 
	\end{enumerate}
We define the adjacency matrix $A$ by
$$
A = H + f \f e \f e^*\,,
$$
where $\f e \deq N^{-1/2}(1,1,\dots,1)^*$, and $0\leq f \leq q$.
\end{definition}

We define the spectral domain
\begin{equation*}
{\bf S} \equiv{\bf S}\deq \{E+\ii \eta \in \bb C\col N^{-1}< \eta\leq 1\}\,.
\end{equation*}
We always use the notation $z = E + \ii \eta$ for the real and imaginary parts of a spectral parameter $z$.
For $\im z \ne 0$, we define the Stieltjes transform $m$ of the semicircle density $\varrho$ by 
\begin{equation*} 
 m(z)\deq \int \frac{\varrho(\dd x) }{x-z} \,, \qquad \varrho(\dd x) \deq \frac{1}{2\pi} \sqrt{(4-x^2)_+}\,\dd x\,,
\end{equation*}
which is characterized as the unique solution of
\begin{equation*}
m(z) + \frac{1}{ m(z)} + z =0
\end{equation*}
satisfying $\im m(z) \im z>0$. We recall the Green function $G$ from \eqref{def_G}.
Finally we define the fundamental error parameter
\begin{equation} \label{zeta}
\zeta\equiv\zeta(N,r,q,\eta,f)\deq \pbb{\frac{r}{q^2}}^{1/4}+\frac{r}{(N\eta)^{1/6}}+\frac{r}{(\log \eta+\log N)q}+\frac{f}{(N\eta)^{1/4}}\,.
\end{equation}
We now state our main result.

\begin{theorem}[Local law]\label{maintheorem}
There is a universal constant $C_* \geq 1$ such that the following holds.
	Let $A$ be defined as in Definition \ref{def:sperse} and $G$ be its Green function \eqref{def_G}. Let $z \in \f S$. Let $r\geq 10$, $1\leq q \leq N^{1/2}$, and $t>0$. Suppose that
	\begin{equation} \label{t_zeta}
	t\zeta\leq 1\,,
	\end{equation} 
then
	\begin{equation} \label{semiH}
	\bb P \pB{\max_{i,j}|G_{ij}(z)-m(z) \delta_{ij}|>t\zeta} \leq  N^5 \pbb{\frac{C_*}{t}}^r\,.
	\end{equation}
\end{theorem}

\begin{remark}	\label{1.4}	
	The constant $C_*$ in \eqref{semiH} can be chosen to be $1000$; see \eqref{O4.44} below. This numerical value can be improved by more careful estimates; we shall not pursue this here.
\end{remark}

We draw several consequences from Theorem \ref{maintheorem}. For all of them, we assume the upper bound $q \leq (\log N)^{10}$, which can however be relaxed to $q \leq N^{1/2}$ without much sweat; it originates from the last term of \eqref{zeta}, which is a naive estimate of the contribution of the expectation $f\f e \f e^*$ of $A$. This upper bound can be removed by a more careful treatment presented in Section 7 of \cite{EKYY1}. However, since in this paper we are interested in the regime $q \leq (\log N)^3$ not covered by the results of \cite{EKYY1}, we shall not do so.

We obtain a local law under the condition $q \geq C\sqrt{\log N}$, i.e.\ $pN \geq C \log N$. More precisely, by setting $r=\log N$ and $t=C_*\ee^{5+D}$ for some $D>0$ (see Appendix \ref{sec:appendix} for details), we deduce from \eqref{semiH} that for any $D>0,\delta \in (0,1)$ there exist $C\equiv C(\delta,D), N_0\equiv N_0(\delta,D)>0$ such that
\begin{equation} \label{semi} 
\bb P\qB{\max_{i,j}|G_{ij}(z)-m\delta_{ij}(z)|\leq \delta} \geq  1-N^{-D}
\end{equation}
whenever $\eta \geq N^{-1+(\log N)^{-1/2}}$, $C\sqrt{\log N} \leq q \leq (\log N)^{10}$ and $N \geq N_0$. Explicitly, we can choose
\begin{equation} \label{expl_constants}
C(\delta,D)=\bigg(\frac{4C_*\ee ^{5+D}}{\delta}\bigg)^2\,, \qquad N_0(\delta, D) = \exp \qBB{\pbb{\ee^{10}\log \pbb{\frac{4 C_* \ee^{5 + D}}{\delta}}}^2}\,.
\end{equation}

By \eqref{semi} and a standard complex analysis argument (e.g.\ see Section 8 of \cite{BK16}) we have the following result.

\begin{theorem} [Local law for density of states] \label{thm:ind}
	Let $\mu \deq \frac{1}{N} \sum_{i = 1}^N \delta_{\lambda_i(A)}$ be the empirical eigenvalue density of $A$. For $D>0,\delta \in (0,1)$ there exist constants $C\equiv C(\delta,D), N_0\equiv N_0(\delta,D)>0$, given by \eqref{expl_constants}, such that for any interval $I \subset \R$ satisfying $|I| \geq N^{-1+(\log N)^{-1/2}}$, we have
\begin{equation*}
\bb P\Big[|\mu(I)-\varrho(I)|\leq \delta|I|\Big] \geq  1-N^{-D}
\end{equation*}
whenever $C\sqrt{\log N} \leq q \leq (\log N)^{10}$ and $N \geq N_0$.
\end{theorem}

\begin{remark}
Previously, in \cite{TVW} a local law for the density of states under the assumption $q \to \infty$ was proved down to scales $\abs{I} \geq \pb{\frac{\log q}{q}}^{1/5}$. This scale is too large to distinguish the presence of isolated vertices, as explained in the introduction. Very recently, in \cite{DZ18} this result was improved to $\abs{I} \geq \frac{\log N}{q^2}$ in the supercritical regime, under the assumption that $I$ is separated from the spectral edges.
\end{remark}

\begin{remark}
Theorem \ref{thm:ind} says nothing about the locations of the extreme eigenvalues. This question was addressed in \cite{BBK1,BBK2}, where it was proved that the extreme eigenvalues converge to the spectral edges in the supercritical regime, and they become outliers in the subcritical regime.
\end{remark}

A standard consequence of Theorem \ref{maintheorem} (e.g.\ in \cite[Theorem 2.10]{BK16}) is the following.

\begin{theorem}[Complete eigenvector delocalization] \label{thm:deloc}
For any $D>0$ there exists $C\equiv C(D)>0, N_0 \equiv N_0(D)$ such that for $q \geq C \sqrt{\log N}$ and $N \geq N_0(D)$,
\begin{equation*}
\p \pbb{\exists i, \norm{\f u_i}_\infty \geq N^{(\log N)^{-1/2}- 1/2}} \leq N^{-D}\,,
\end{equation*} 
where $\f u_i \in \bb S^{N - 1}$ is the $i$-th normalized eigenvector of $A$. (We can take $C(D) = C(1,D)$ and $N_0(D) = N_0(1,D)$ in \eqref{expl_constants}.)
\end{theorem}

\begin{remark}
Previously, in \cite{TVW} it was proved that with high probability $\norm{\f u_i}_\infty = o(1)$ for $q \geq C \sqrt{\log N}$. For eigenvectors associated with eigenvalues separated from the spectral edges, the authors obtained the stronger estimate $\norm{\f u_i}_\infty = O \pb{(\log (q/\sqrt{\log N}))^{1.1} \sqrt{\log N} / q}$. Very recently, in \cite{DZ18} this was improved to $\norm{\f u_i}_\infty = O(\sqrt{\log N} / q)$ for eigenvalues separated from the spectral edges.
\end{remark}

Moreover, we obtain the following probabilistic version of local quantum unique ergodicity for the Erd\H{o}s-R\'enyi random graph, by combining Theorem \ref{thm:deloc} with \cite[Proposition 8.3]{BKY15}.

\begin{theorem}[Probabilistic local quantum unique ergodicity] \label{thm:QUE}
Let $A$ be the rescaled adjacency matrix of the Erd\H{o}s-R\'enyi graph $G(N,p)$. Let $a_1, \dots, a_N \in \R$ be deterministic numbers satisfying $\sum_{i = 1}^N a_i = 0$.
For any $D>0$ there exists $C\equiv C(D)>0, N_0 \equiv N_0(D)$ such that for $pN \geq C^2 \log N$, $N \geq N_0(D)$, $k = 1, \dots, N$, and $\theta \geq 1$
\begin{equation*}
\sum_{i = 1}^N a_i u_k(i)^2 = O \pbb{\frac{\theta^2 N^{(\log N)^{-1/2}}}{N} \pbb{\sum_{i = 1}^N a_i^2}^{1/2}}
\end{equation*}
with probability at least $1 - N^{-D} - \ee^{-\theta \sqrt{\log \theta}}$. Here $u_k(i)$ is the $i$-th component of the $k$-th eigenvector of $A$.
\end{theorem}

Theorem \ref{thm:QUE} states that, on deterministic sets of at least $\pb{\theta^2 N^{(\log N)^{-1/2}}}^2$ vertices, 
all eigenvectors of the random graph $A$ are completely flat with high probability. In other words, with high probability, the random probability measure $i \mapsto v_k(i)^2$
is close (when tested against deterministic test functions) to the uniform probability measure $i \mapsto 1/N$ on $\{1, \dots, N\}$. 
For instance, let $I \subset \{1, \dots, N\}$ be a deterministic subset of vertices. Setting $a_i \deq \ind{i \in I} - \abs{I} / N$
in Theorem \ref{thm:QUE}, we obtain
\begin{equation} \label{e:QUE_example}
\sum_{i \in I} v_k(i)^2 \;=\; \sum_{i \in I} \frac{1}{N} + O \pbb{\frac{\theta^2 N^{(\log N)^{-1/2}} \sqrt{\abs{I}}}{N}}
\end{equation}
with probability at least $1 - N^{-D} - \ee^{-\theta \sqrt{\log \theta}}$.
The main term on the right-hand side of \eqref{e:QUE_example} is much larger than the error term provided that $\abs{I} \gg \pb{\theta^2 N^{(\log N)^{-1/2}}}^2$.

\section{Multilinear large deviation bounds for sparse random vectors} \label{sec:LDE}

The rest of this paper is devoted to the proof of Theorem \ref{maintheorem}.
In this section we derive large deviation estimates for multilinear forms of sparse random vectors with independent components. We focus on the linear and bilinear estimates, which are sufficient for our applications. These estimates are designed to be optimal in the regime of very sparse random vectors.

For $n \in \N^*$ denote by $[n] \deq \{1, 2, \dots, n\}$, and for a finite set $A$ we denote by $\abs{A}$ the cardinality of $A$. For two sets $U,V$, we denote by $V^U$ the set of functions from $U$ to $V$.
For $r \geq 1$ we denote by $\norm{X}_r \deq (\E \abs{X}^r)^{1/r}$ the $L^r$-norm of the random variable $X$.

\begin{proposition} \label{prop:lde1}
Let $r$ be even and $1 \leq q \leq N^{1/2}$. Let $X_1, \dots, X_N \in \C$ be independent random variables satisfying
\begin{equation} \label{moment_conditions_X}
\E X_i = 0\,, \qquad \E \abs{X_i}^k \leq \frac{1}{N q^{k - 2}} \quad (2 \leq k \leq r)\,.
\end{equation}
Let $a_1, \dots, a_N \in \C$ be deterministic, and suppose that
\begin{equation*}
\pbb{\frac{1}{N} \sum_i \abs{a_i}^2}^{1/2} \leq \gamma\,, \qquad \frac{\max_i \abs{a_i}}{q} \leq \psi
\end{equation*}
for some $\gamma,\psi \geq 0$.
Then
\begin{equation*}
\normbb{\sum_i a_i X_i}_r \leq \pbb{\frac{2 r}{1 + 2 (\log (\psi / \gamma))_+} \vee 2} (\gamma \vee \psi)\,.
\end{equation*}
\end{proposition}

\begin{proof}
To avoid cumbersome complex conjugates in our notation, we assume for simplicity that all quantities are real-valued. Denote by $\fra P(r)$ the set of partitions of $[r]$, and by $\fra P_{\geq 2}(r)$ the subset of $\fra P(r)$ whose blocks all have size at least two. For $(i_1, \dots, i_r) = [N]^r$ denote by $P(i_1, \dots, i_r) \in \fra P(r)$ the partition of $[r]$ associated with the equivalence relation $k \sim l$ if and only if $i_k = i_l$. Using the identity
\begin{equation*}
1 = \sum_{\Pi \in \fra P(r)} \ind{P(i_1, \dots, i_r) = \Pi}  = \sum_{\Pi \in \fra P(r)} \sum_{s \in [N]^\Pi} \prod_{\pi \neq \tilde \pi \in \Pi} \ind{s_\pi \neq s_{\tilde \pi}} \prod_{\pi \in \Pi} \prod_{k \in \pi} \ind{i_k = s_\pi}\,,
\end{equation*}
the independence of the variables $X_i$, and the fact that $\E X_i = 0$ for all $i$, we find
\begin{equation*}
\normbb{\sum_i a_i X_i}_r^r = \sum_{i_1, \dots, i_r = 1}^N a_{i_1} \cdots a_{i_r} \, \E X_{i_1} \cdots X_{i_r}
\leq \sum_{\Pi \in \fra P_{\geq 2}(r)} \sum_{s \in [N]^\Pi} \prod_{\pi \in \Pi} \abs{a_{s_\pi}}^{\abs{\pi}} \E \abs{X_{s_\pi}}^{\abs{\pi}}\,.
\end{equation*}
Here the restriction $\Pi \in \fra P_{\geq 2}(r)$ follows from the fact that $\E X_i = 0$, and we dropped the indicator function $\prod_{\pi \neq \tilde \pi \in \Pi} \ind{s_\pi \neq s_{\tilde \pi}}$ to obtain an upper bound.
We deduce from
\begin{multline*}
\sum_{s \in [N]^\Pi} \prod_{\pi \in \Pi} \abs{a_{s_\pi}}^{\abs{\pi}} \E \abs{X_{s_\pi}}^{\abs{\pi}} = \prod_{\pi \in \Pi} \sum_{s = 1}^N \abs{a_s}^{\abs{\pi}} \E \abs{X_s}^{\abs{\pi}}
\\
\leq \prod_{\pi \in \Pi} \pbb{\sum_i \abs{a_i}^2} (\max_i \abs{a_i})^{\abs{\pi} - 2} \frac{1}{N q^{\abs{\pi} - 2}} \leq \gamma^{2 \abs{\Pi}} \psi^{r - 2 \abs{\Pi}}
\end{multline*}
that
\begin{equation*}
\normbb{\sum_i a_i X_i}_r^r \leq R_r(\gamma,\psi)\,, \qquad R_r(\gamma,\psi) \deq \sum_{k = 1}^{r/2} S(r,k) \gamma^{2 k} \psi^{r - 2 k}\,,
\end{equation*}
where $S(r,k) \deq \abs{\h{\Pi \in \fra P(r) \col \abs{\Pi} = k}}$ is the Stirling number of the second kind. To conclude the proof, it therefore suffices to prove that
\begin{equation} \label{comb_estimate}
R_r(\gamma,\psi)^{1/r} \leq \pbb{\frac{2 r}{1 + 2 (\log (\psi / \gamma))_+} \vee 2} (\gamma \vee \psi)\,.
\end{equation}

Using the bound $S(r,k) \leq \frac{1}{2} \binom{r}{k} k^{r - k}$, we conclude
\begin{equation*}
R_r(\gamma,\psi) \leq 2^r \max_{1 \leq k \leq r/2} f(k) \,, \qquad f(k) \deq k^{r - k} \gamma^{2k} \psi^{r - 2k}\,.
\end{equation*}
If $\psi \leq \gamma$ then we estimate $f(k)^{1/r} \leq r \gamma$ and \eqref{comb_estimate} follows.

It suffices therefore to assume $\gamma \leq \psi$. We consider the function $f$ on the interval $(0,\infty)$. By differentiation, we find that $\log f$ is concave on $(0,\infty)$ and maximized for
\begin{equation} \label{max_f}
\frac{r}{k} = 1 + \log k + 2 \log (\psi / \gamma)\,.
\end{equation}
Since the left-hand side of \eqref{max_f} is decreasing and its right-hand side increasing, it is easy to see that  \eqref{max_f} has a unique solution $k_*$ in $(0,\infty)$. If $k_* < 1$ then, by concavity of $\log f$, we find that $f$ is decreasing on $[1,\infty)$ and
\begin{equation*}
R_r(\gamma,\psi)^{1/r} \leq 2 f(1)^{1/r} \leq 2 \psi,
\end{equation*}
which is \eqref{comb_estimate}. Let us therefore assume that $k_* \geq 1$. Then we find
\begin{equation*}
R_r(\gamma,\psi)^{1/r} \leq 2 f(k_*)^{1/r} \leq 2 k_* \psi\,.
\end{equation*}
Moreover, for $k_* \geq 1$ from \eqref{max_f} we find that
\begin{equation*}
k_* \leq \frac{r}{1 + 2 \log (\psi / \gamma)}\,,
\end{equation*}
and \eqref{comb_estimate} follows.
\end{proof}


\begin{proposition}\label{prop:critical_bound}
Let $r$ be even and $1 \leq q \leq N^{1/2}$. Let $X_1, \dots, X_N \in \C$ be independent random variables satisfying \eqref{moment_conditions_X}.
For any deterministic $a_1, \dots, a_N \in \C$ we have
\begin{equation*}
\normbb{\sum_i a_i \pb{\abs{X_i}^2 - \E \abs{X_i}^2}}_r \leq 2 \pbb{1 + \frac{2q^2}{N}} \max_i \abs{a_i} \pbb{\frac{r}{q^2} \vee \sqrt{\frac{r}{q^2}}}\,.
\end{equation*}
\end{proposition}

\begin{proof}
Since $\E \abs{X_i}^2 \leq \frac{1}{N}$, a simple estimate using the binomial theorem yields
\begin{equation*}
\E \pB{\abs{X_i}^2 - \E \abs{X_i}^2}^r \leq \frac{1}{N q^{2r - 2}} \pbb{1 + \frac{2q^2}{N}}^r\,.
\end{equation*}
As in the proof of Proposition \ref{prop:lde1}, we conclude that
\begin{align*}
\normbb{\sum_i a_i \pb{\abs{X_i}^2 - \E \abs{X_i}^2}}_r^r &\leq \sum_{k = 1}^{r/2} S(r,k) (\max_i \abs{a_i})^r \frac{1}{q^{2r - 2k}} \pbb{1 + \frac{2q^2}{N}}^r
\\
&\leq \qbb{2 \max_i \abs{a_i} \pbb{1 + \frac{2q^2}{N}}}^r \max_{1 \leq k \leq r/2} \pbb{\frac{k}{q^2}}^{r - k}
\end{align*}
and the claim follows.
\end{proof}

\begin{proposition} \label{prop:lde2}
Let $r$ be even and $1 \leq q \leq N^{1/2}$.
Let $X_1, \dots, X_N, Y_1, \dots, Y_N$ be independent random variables, all satisfying \eqref{moment_conditions_X}. Let $a_{ij} \in \C$, $i,j = 1, \dots, N$ be deterministic, and suppose that
\begin{equation*}
\pbb{\max_i \frac{1}{N} \sum_{j} \abs{a_{ij}}^2}^{1/2} \vee \pbb{\max_j \frac{1}{N} \sum_{i} \abs{a_{ij}}^2}^{1/2} \leq \gamma \,, \qquad  \frac{\max_{i,j} \abs{a_{ij}}}{q^2} \leq \psi
\end{equation*}
for some $\gamma, \psi \geq 0$. Then
\begin{equation} \label{def_gamma_psi_2}
\normbb{\sum_{i,j} a_{ij} X_i Y_j}_r \leq \pbb{\frac{2 r}{1 + (\log (\psi / \gamma))_+} \vee 2}^2 (\gamma \vee \psi)\,.
\end{equation}
\end{proposition}
\begin{proof}
As in the proof of Proposition \ref{prop:lde1}, to simplify notation we assume that all quantities are real-valued. We write
\begin{equation*}
\normbb{\sum_{i,j} a_{ij} X_i Y_j}_r^r = \sum_{i_1, \dots, i_r} \sum_{j_1, \dots, j_r} a_{i_1 j_1} \cdots a_{i_r j_r} \E X_{i_1} Y_{j_1} \cdots X_{i_r}  Y_{j_r}\,.
\end{equation*}
Analogously to the proof of Proposition \ref{prop:lde1}, we encode the terms on the right-hand side using two partitions $\Pi_1, \Pi_2 \in \fra P_{\geq 2}(r)$, by requiring that $k,l \in [r]$ are in the same block of $\Pi_1$ if and only if $i_k = i_l$, and $k,l \in [r]$ are in the same block of $\Pi_2$ if and only if $j_k = j_l$. Each block of $\Pi_1$ and $\Pi_2$ has size at least two. We encode each pair $(\Pi_1, \Pi_2) \in \fra P_{\geq 2}(r)^2$ using a multigraph $G(\Pi_1, \Pi_2)$, whose vertex set is $\Pi_1 \sqcup \Pi_2$ and whose set of edges is obtained by adding, for each $k = 1, \dots, r$, an edge $\{\pi_1, \pi_2\}$ where $k \in \pi_1 \in \Pi_1$ and $k \in \pi_2 \in \Pi_2$.

It is immediate that $G(\Pi_1,\Pi_2)$ has $r$ edges, that each vertex has degree at least two, and that $G(\Pi_1,\Pi_2)$ is bipartite. Thus we find
\begin{equation} \label{lde_est_exp}
\normbb{\sum_{i,j} a_{ij} X_i Y_j}_r^r \leq \sum_{\Pi_1, \Pi_2 \in \fra P_{\geq 2}(r)} \sum_{s \in [N]^{\Pi_1 \sqcup \Pi_2}} \pBB{\prod_{\{\pi_1, \pi_2\} \in E(G(\Pi_1, \Pi_2))} \abs{a_{s_{\pi_1} s_{\pi_2}}}} \pBB{\prod_{\pi \in \Pi_1 \sqcup \Pi_2} \frac{1}{N q^{\abs{\pi} - 2}}}\,.
\end{equation}
To estimate the right-hand side of \eqref{lde_est_exp}, we construct an algorithm that successively sums out the variables $s_\pi$. For a multigraph $G = (V(G), E(G))$ define
\begin{equation*}
\val(G) \deq \frac{1}{N^{\abs{V(G)}}} \sum_{s \in [N]^{V(G)}} \pBB{\prod_{\{\pi_1, \pi_2\} \in E(G)} \abs{a_{s_{\pi_1} s_{\pi_2}}}}\,
\end{equation*}
so that \eqref{lde_est_exp} becomes
\begin{equation}
\normbb{\sum_{i,j} a_{ij} X_i Y_j}_r^r \leq \sum_{\Pi_1, \Pi_2 \in \fra P_{\geq 2}(r)} \frac{1}{q^{2r - 2 \abs{\Pi_1} - 2 \abs{\Pi_2}}} \val (G(\Pi_1, \Pi_2))
\end{equation}

We shall need the following notions.
We denote by $x(G) \deq \sum_{v \in V(G)} (\deg_G(v) - 2)_+$ the total excess degree of a multigraph $G$.  We fix $\Pi_1, \Pi_2 \in \fra P_{\geq 2}(r)$ and abbreviate $k \deq \abs{\Pi_1} + \abs{\Pi_2}$. Abbreviate
\begin{equation*}
a \deq \pbb{\max_i \frac{1}{N} \sum_{j} \abs{a_{ij}}^2}^{1/2} \vee \pbb{\max_j \frac{1}{N} \sum_{i} \abs{a_{ij}}^2}^{1/2}\,, \qquad A \deq \max_{i,j} \abs{a_{ij}}\,.
\end{equation*}

We construct a sequence $G_0, G_1, \dots, G_L$ of multigraphs recursively as follows. First, $G_0 \deq G(\Pi_1, \Pi_2)$. Since $\deg_{G_0}(v) \geq 2$ for all $v \in V(G_0)$, we immediately find $x(G_0) = 2r - 2k$. For $\ell \geq 0$, the multigraph $G_{\ell+1}$ is constructed from $G_\ell$ according to the two following cases.

\begin{itemize}
\item[1.]
$G_\ell$ has no vertex of degree $\leq 2$. Choose an arbitrary $e \in E(G_{\ell})$ and let $G_{\ell + 1}$ be the multigraph obtained from $G_\ell$ by removing the edge $e$. Clearly,
\begin{equation*}
x(G_{\ell+1}) = x(G_{\ell}) - 2\,, \qquad \val(G_{\ell}) \leq A \val(G_{\ell + 1})\,.
\end{equation*}

\item[2.]
$G_\ell$ has a vertex $v$ of degree $\leq 2$. Let $G_{\ell + 1}$ be the multigraph obtained from $G_\ell$ by removing the vertex $v$ and all $\deg_{G_{\ell}}(v)$ edges incident to it. Clearly,
\begin{equation*}
x(G_{\ell+1}) \leq x(G_{\ell})\,, \qquad \val(G_{\ell}) \leq a^{\deg_{G_\ell}(v)} \val(G_{\ell + 1})\,,
\end{equation*}
where the second estimate follows by Cauchy-Schwarz and the fact that $\deg_{G_\ell}(v) = 0,1,2$.
\end{itemize}

It is easy to see that the algorithm terminates at $G_L = (\emptyset, \emptyset)$, the empty graph with $\val(G_L) = 1$.

Next, for $i = 1,2$, let $r_i$ be the total number of edges removed by the algorithm in all steps of case $i$. Clearly, $r_1 + r_2 = r$. Moreover, since $x(G_L) \geq 0$, we find that $x(G_0) - 2 r_1 \geq 0$, i.e.\ $r_1 \leq r - k$, which implies $r_2 \geq k$. We conclude that
\begin{equation*}
\val(G_0) \leq A^{r_1} a^{r_2} \leq A^{r - k} a^k\,,
\end{equation*}
where we used that $a \leq A$. Hence, with $k_i \deq \abs{\Pi_i}$ (so that $k = k_1 + k_2$), we have
\begin{equation*}
\frac{1}{q^{2r - 2k_1 - 2k_2}} \val (G(\Pi_1, \Pi_2)) \leq \pbb{\frac{A}{q^2}}^{r - k_1 - k_2} a^{k_1+k_2}\,.
\end{equation*}
Summarizing, we have
\begin{equation*}
\normbb{\sum_{i,j} a_{ij} X_i Y_j}_r^r \leq \sum_{k_1,k_2 = 1}^{r/2} S(r,k_1) S(r,k_2) \pbb{\frac{A}{q^2}}^{r - k_1 - k_2} a^{k_1+k_2} \leq \pBB{\sum_{k = 1}^{r/2}S(r,k) \sqrt{\gamma}^{\, 2k} \sqrt{\psi}^{\, r - 2k}}^2\,.
\end{equation*}
The claim now follows from applying the estimate \eqref{comb_estimate} to $R_r(\sqrt{\gamma}, \sqrt{\psi})$.
\end{proof}

\begin{remark}
It might be tempting to try to prove an analogous result with the parameter $\gamma$ in \eqref{def_gamma_psi_2} replaced with the smaller $\frac{1}{N^2} \sum_{i,j} \abs{a_{ij}}^2$. However, in this case the moment-based argument used in the proof of Proposition \ref{prop:lde2} does not lead to a prefactor that is small in the targeted regime $q^2 \gtrsim r$, $\max_{i,j} \abs{a_{ij}} \asymp 1$, $\max_i \frac{1}{N} \sum_{j}\abs{a_{ij}}^2 \asymp N^{-c}$, which corresponds to the typical behaviour of the parameters $q,r,a_{ij}$ in the applications to proving a local law for supercritical graphs (see Section \ref{sec:proof} below).  Indeed, consider the case where $\Pi_1$ is arbitrary and $\Pi_2$ consists of a single block. There are an order $(Cr)^r$ partitions $\Pi_1$. If we only allow estimates in terms of $\frac{1}{N^2} \sum_{i,j} \abs{a_{ij}}^2$ and $\max_{i,j} \abs{a_{ij}}$, then each such partition yields a contribution of order
\begin{equation*}
\frac{1}{N^2}\sum_{i,j} \abs{a_{ij}}^2 \pbb{\frac{\max_{i,j} \abs{a_{ij}}}{q}}^{r-2}
\end{equation*}
Thus, the total contribution of such pairings is of the order $N^{-c} q^2 \pB{\frac{Cr}{q}}^{r}$, which is much too large.

This observation leads us believe that $\max_i \frac{1}{N} \sum_{j} \abs{a_{ij}}^2$ is the correct quantity, instead of the smaller $\frac{1}{N^2} \sum_{i,j} \abs{a_{ij}}^2$, when deriving large deviation estimates in the very sparse regime.
\end{remark}

\begin{proposition} \label{prop:lde3}
Let $r$ be even and $1 \leq q \leq N^{1/2}$.
Let $X_1, \dots, X_N$ be independent random variables satisfying \eqref{moment_conditions_X}. Let $a_{ij} \in \C$, $i,j = 1, \dots, N$ be deterministic, and suppose that
\begin{equation*}
\pbb{\max_i \frac{1}{N} \sum_{j} \abs{a_{ij}}^2}^{1/2} \vee \pbb{\max_j \frac{1}{N} \sum_{i} \abs{a_{ij}}^2}^{1/2} \leq \gamma \,, \qquad  \frac{\max_{i,j} \abs{a_{ij}}}{q^2} \leq \psi
\end{equation*}
for some $\gamma, \psi \geq 0$. Then
\begin{equation} \label{def_gamma_psi_3}
\normbb{\sum_{i \neq j} a_{ij} X_i X_j}_r \leq \pbb{\frac{4 r}{1 + (\log (\psi / \gamma))_+} \vee 4}^2 (\gamma \vee \psi)\,.
\end{equation}
\end{proposition}

\begin{proof}
Using the simple decoupling argument from the proof of \cite[Lemma B.4]{EKYY3}, we find
\begin{equation*}
\normbb{\sum_{i \neq j} a_{ij} X_i X_j}_r \leq \frac{1}{2^{N - 2}} \sum_{I \sqcup J = [N]} \normbb{\sum_{i \in I} \sum_{j \in J} a_{ij} X_i X_j}_r\,,
\end{equation*}
where on the right-hand side the sum ranges over disjoint nonempty sets $I,J$ whose union is $[N]$. The $L^r$-norm on the right-hand side can be estimated using Proposition \ref{prop:lde2}, and the sum $\sum_{I \sqcup J = [N]}$ yields a factor $2^N - 2$. Hence the claim follows.

Alternatively, we can repeat the proof of Proposition \ref{prop:lde2} almost verbatim, except that instead of having a bipartite multigraph encoded by two partitions, $\Pi_1, \Pi_2 \in \fra P_{\geq 2}(r)$, we have a multigraph without loops encoded by the single partition $\Pi \in \fra P_{\geq 2}(2r)$.
\end{proof}

\section{Proof of Theorem \ref{maintheorem}}  \label{sec:proof}

In this section we give the proof of Theorem \ref{maintheorem}. Throughout this section we frequently omit the spectral parameter $z$ from our notation, unless this leads to confusion. We start with a few standard tools.

\begin{definition}
	For $k \in [N]$ we define $A^{(k)}$ as the $(N-1)\times(N-1)$ matrix
	\begin{equation}
	A^{(k)}=(A_{ij})_{i,j\in [N] \setminus\{k\}}\,.
	\end{equation}
	Moreover, we define the Green function of $A^{(k)}$ through
	\begin{equation}
	G^{(k)}(z)=(A^{(k)}-z)^{-1}\,.
	\end{equation}
	We also abbreviate
	\begin{equation}
	\sum_{i}^{(k)}\deq\sum_{i\col i\ne k}\,.
	\end{equation}
\end{definition}

By applying the resolvent identity to $G - G^{*}$, we get the Ward identity
\begin{equation} \label{Ward}
\sum_j |G_{ij}|^2 \;=\; \frac{\im G_{ii}}{\eta} \,,
\end{equation}
which of course also holds for $G^{(k)}$ for any $k \in [N]$.

The following variant of Schur's complement formula is standard and its proof is given e.g.\ in \cite{BK16}.

\begin{lemma}\label{Schur}
For $i \neq j$ we have
\begin{equation} \label{schurix}
G_{ij}=-G_{jj}\sum_{k}^{(j)}G^{(j)}_{ik}A_{kj} = -G_{ii}\sum_{k}^{(i)}A_{ik}G^{(i)}_{kj}\,.
\end{equation}
For $i,j \neq k$ we have
\begin{equation} \label{O3.11}
G_{ij}^{(k)}=G_{ij} - \frac{G_{ik} G_{kj}}{G_{kk}}\,.
\end{equation}
Finally,
\begin{equation} \label{SPSF} 
\ff{G_{ii}} = A_{ii}-z-\sum_{k,l}^{(i)}A_{ik}G^{(i)}_{kl}A_{l i}
\end{equation} 
for all $i$.
\end{lemma}

Define the Stieltjes transform of the empirical eigenvalue measure as
\begin{equation*}
s \deq \frac{1}{N} \sum_{i = 1}^N G_{ii}\,.
\end{equation*}
From \eqref{SPSF} we easily deduce the following result, which serves as the starting point of our analysis of the diagonal entries of $G$.

\begin{lem} \label{lem11}
We have 
\begin{equation} \label{O4.1}
\ff{G_{ii}} = -z-s+Y_i\,,
\end{equation}
where 
\begin{multline} \label{def_Y_A_Z}
Y_i \deq H_{ii}+\ff{N}\sum_k\frac{G_{ki}G_{ik}}{G_{ii}}-\sum_{k\ne l}^{(i)}H_{ik}G^{(i)}_{kl}H_{l i}-\sum_{k}^{(i)}\Big(H_{ik}^2-\frac{1}{N}\Big)G_{kk}^{(i)}\\
+
\frac{f}{N}-\Big(\frac{f}{N}\Big)^2\sum_{k,l}^{(i)}G_{kl}^{(i)}-\frac{f}{N}\sum_{k,l}^{(i)}G_{kl}^{(i)}(H_{ik}+H_{li})\,.
\end{multline}
\end{lem}
\begin{proof}
	The proof consists of comparing the right-hand side of \eqref{SPSF} to its conditional expectation $\bb E (\cdot\,|H^{(i)})$, and we omit the details.
\end{proof}

Next, we define the random $z$-dependent error parameter
$$\Gamma \deq \max_{i,j}|G_{ij}| \vee \max_{i,j\ne k}|G_{ij}^{(k)}|\,,$$ 
as well as the $z$-dependent indicator function
\begin{equation}
\phi \deq {\bf 1}(\Gamma\leq 2)\,.
\end{equation}

The proof consists of a stochastic continuity argument, which establishes control on the entries of $G$ under the bootstrapping assumption $\phi = 1$. The following lemma provides the key probabilistic estimates that allow us to estimate the various fluctuating error terms that appear in the proof.

\begin{lemma}[Main estimates]\la{lemma:main_estimates} 
Let $2\leq r\leq q^2\leq N$. Then
\begin{equation} \label{O4.3}
\max_i\|\phi Y_iG_{ii}\|_r \leq 48^2\bigg[\Big(\frac{r}{q^2}\Big)^{1/2}+\frac{r^2}{(N\eta)^{1/3}}+\bigg(\frac{r}{(\log N+\log \eta)q}\bigg)^2+\frac{f^2}{\sqrt{N\eta}}\bigg]\,,
\end{equation}
\begin{equation} \label{O4.4}
\max_{i\ne j}\|\phi\, G_{ij}\|_r\leq 12\bigg[\,\frac{1}{q}+\frac{r}{(N\eta)^{1/6}}+\frac{r}{(\log N+\log \eta)q}+\frac{f}{(N\eta)^{1/4}}\bigg]\,,
\end{equation}
and
\begin{equation} \label{O4.66}
\max_{i,j\ne k} \big\|\phi(G_{ij}-G_{ij}^{(k)})\big\|_r \leq 12\bigg[\,\frac{1}{q}+\frac{r}{(N\eta)^{1/6}}+\frac{r}{(\log N+\log \eta)q}+\frac{f}{(N\eta)^{1/4}}\bigg]
\end{equation}
for all $z\in \bf S$.
\end{lemma}

\begin{proof}
We begin with \eqref{O4.3}. Pick $i \in [N]$. We shall first bound $\|\phi Y_i\|_r$ by estimating the $L^r$-norm of the terms on the right-hand side of \eqref{def_Y_A_Z}. Definition \ref{def:sperse} (iii) ensures
\begin{equation} \label{O4.10}
\|H_{ii}\|_{r} \leq 1/q\,,
\end{equation} 
and by the Ward identity \eqref{Ward} we get
\begin{equation} \label{O4.8}
\Big|\ff{N}\sum_k\frac{G_{ki}G_{ik}}{G_{ii}}\Big|\leq \frac{1}{N \eta}\,.
\end{equation}

We shall estimate the $L^r$-norms of the remaining terms on the right-hand side of \eqref{def_Y_A_Z} using the multilinear large deviation estimates from Section \ref{sec:LDE}. We begin with the term $\sum_{k\ne l}^{(i)}H_{ik}G^{(i)}_{kl}H_{l i}$. Define
\begin{equation*}
\phi^{(i)} \deq \indB{\max_{k,l \neq i} \abs{G_{kl}^{(i)}} \leq 2}\,,
\end{equation*}
and note that $\phi \leq \phi^{(i)}$. Denote by $\norm{\,\cdot\,}_{r | H^{(i)}}$ the conditional $L^r$ norm with respect to the conditional expectation $\E[ \,\cdot\, | H^{(i)}]$. Then by the nesting property of conditional $L^r$-norms, we have
\begin{equation*}
\normBB{\phi \sum_{k\ne l}^{(i)}H_{ik}G^{(i)}_{kl}H_{l i}}_r \leq \normBB{\phi^{(i)} \sum_{k\ne l}^{(i)}H_{ik}G^{(i)}_{kl}H_{l i}}_r
\\
= \normBB{\phi^{(i)} \normBB{\sum_{k\ne l}^{(i)}H_{ik}G^{(i)}_{kl}H_{l i}}_{r | H^{(i)}}}_r\,.
\end{equation*}
The conditional $L^r$ norm is amenable to the large deviation estimate from Proposition \ref{prop:lde3}. To that end, we use the Ward identity to estimate
\begin{equation}
\phi^{(i)}\,\Big(\max_k \frac{1}{N}\sum_{l}\big|G_{kl}^{(i)}\big|^2\Big)^{1/2}\vee \phi\,\Big(\max_l \frac{1}{N}\sum_{k}\big|G_{kl}^{(i)}\big|^2\Big)^{1/2} \leq  \sqrt{\frac{2}{N\eta}}\,,
\end{equation}
so that Proposition \ref{prop:lde3} with the choices $\gamma = \sqrt{\frac{2}{N \eta}}$ and $\psi = \frac{2}{q^2}$ gives
\begin{equation} \label{O4.6}
\begin{aligned}
\normBB{\phi\sum_{k\ne l}^{(i)}H_{ik}G^{(i)}_{kl}H_{l i}}_{r} &\leq \bigg[\pbb{\frac{4 r}{1 + (\log (\sqrt{N\eta} / q^2))_+}}^2+16\bigg] \bigg(\sqrt{\frac{2}{N\eta}} \vee \frac{2}{q^2}\bigg)\\
&\leq 16r^2\cdot\frac{2}{(N\eta)^{1/3}}+\Big(\frac{24r}{\log N+\log \eta}\Big)^2\cdot \frac{2}{q^2}+\frac{32}{\sqrt{N\eta}}+\frac{32}{q^2}\,,
\end{aligned}
\end{equation}
where the second step is obtained by considering the two cases $q^2 \geq (N\eta)^{1/3}$ and $q^2 \leq (N\eta)^{1/3}$ separately.

Similarly, using Proposition \ref{prop:critical_bound} on the fourth term on the right-hand side of \eqref{def_Y_A_Z} gives
\begin{equation} \label{O4.111}
\normBB{\phi \sum_k^{(i)}\pbb{H_{ik}^2-\ff{N}}G_{kk}^{(i)}}_r \leq  4\pbb{1 + \frac{2q^2}{N}}  \sqrt{\frac{r}{q^2}}\leq 12\sqrt{\frac{r}{q^2}} \,.
\end{equation} 

Next, by the Ward identity we have
\begin{equation}
\bigg|\phi\Big(\frac{f}{N}\Big)^2\sum_{k,l}^{(i)}G^{(i)}_{kl}\bigg| \leq f^2\sqrt{\frac{2}{N\eta}}\,.
\end{equation} 

Next, we estimate the last term of \eqref{def_Y_A_Z}. With the abbreviations
\[
a_k\deq \phi^{(i)} \frac{f}{N}\sum_l^{(i)} G_{kl}^{(i)} \,, \qquad X_k\deq  H_{ik}
\]
we have $\max_k|a_k|\leq \sqrt{2}f/\sqrt{N\eta}$ by the Ward identity. From Proposition \ref{prop:lde1} with $\gamma = \psi = \sqrt{2}f/\sqrt{N\eta}$ we therefore get
\begin{equation} \label{O4.13}
\normBB{\phi\frac{f}{N} \sum_{k,l}^{(i)}G_{kl}^{(i)}H_{ik}}_r = \normBB{\normBB{\sum_k^{(i)}a_kX_k}_{r | H^{(i)}}}_r \leq 2r\cdot \frac{\sqrt{2}f}{\sqrt{N\eta}}\leq (r^2+f^2)\sqrt{\frac{2}{N\eta}}\,,
\end{equation}
and the same bound can also be derived for the $L^r$-norm of $\phi\frac{f}{N} \sum_{k,l}^{(i)}G_{kl}^{(i)}H_{li}$.

Summarizing, by Minkowski's inequality and \eqref{O4.10}--\eqref{O4.13}, we get
\begin{multline*}   
\|\phi Y_iG_{ii}\|_r \leq  2 \|\phi Y_i\|_r \leq \frac{2}{q}+\frac{2}{N\eta}+\frac{64r^2}{(N\eta)^{1/3}}+\frac{(48r)^2}{(\log N+\log \eta)^2q^2}+\frac{64}{\sqrt{N\eta}}+\frac{32}{q^2}\\+24 \sqrt{\frac{r}{q^2}}
+\frac{2f}{N}+6f^2\sqrt{\frac{2}{N\eta}}+4r^2\sqrt{\frac{2}{N\eta}}\,.
\end{multline*}
and by $N\eta\geq 1$, $r \geq 2$ we have
\begin{equation} \label{O4.11}
\|\phi Y_iG_{ii}\|_r \leq  \frac{100r^2}{(N\eta)^{1/3}}+\frac{(48r)^2}{(\log N+\log \eta)^2q^2}+100 \sqrt{\frac{r}{q^2}}
+\frac{100f^2}{\sqrt{N\eta}}\,,
\end{equation}
and \eqref{O4.3} follows.

Now we turn to \eqref{O4.4}. For $i \ne j$, by \eqref{schurix} we have
\begin{equation}\la{EqAuxiliere}
G_{ij}=-G_{ii}\sum_{k}^{(i)}\Big(H_{ik}+\frac{f}{N}\Big)G^{(i)}_{kj}\,.
\end{equation}
Thus, for $i \neq j$ we have
\begin{equation*}
\|\phi\, G_{ij}\|_r\leq \normBB{\phi \,{G_{ii}\sum_{k}^{(i)}\Big(H_{ik}+\frac{f}{N}\Big)G^{(i)}_{kj}}}_r\leq 2 \normBB{\phi^{(i)} \sum_{k}^{(i)}H_{ik}G^{(i)}_{kj}}_r +\frac{4f}{\sqrt{N\eta}}\,,
\end{equation*}
where in the last step we estimated the term $\phi G_{ii}\sum_{k}^{(i)} \frac{f}{N} G^{(i)}_{kj}$ by $\frac{4f}{\sqrt{N\eta}}$, using the Ward identity. Invoking Proposition \ref{prop:lde1} with the choices $\gamma = \sqrt{\frac{2}{N \eta}}$ and $\psi = \frac{2}{q}$ yields, by the Ward identity,
\begin{multline*}
\|\phi\, G_{ij}\|_r\leq \bigg(\frac{4r}{1+2(\log(\sqrt{N\eta}/q))_+}+4\bigg)\bigg(\sqrt{\frac{2}{N\eta}}\vee\frac{2}{q}\Big) +\frac{4f}{\sqrt{N\eta}}\\
\leq 4 r\cdot\frac{2}{(N\eta)^{1/6}}+\frac{6 r}{\log N+\log \eta}\cdot \frac{2}{q}+\frac{8}{\sqrt{N\eta}}+\frac{8}{q}+\frac{4f}{\sqrt{N\eta}}\,,
\end{multline*}
where the second step is obtained by considering the two cases $q\geq (N\eta)^{1/6}$ and $q\leq (N\eta)^{1/6}$ separately. Together with $N\eta \geq 1$ we obtain \eqref{O4.4}.

Finally, to show \eqref{O4.66} we use \eqref{O3.11} and \eqref{schurix} to write
\begin{equation*}
G_{ij}^{(k)} - G_{ij} = G_{kj}\sum_l^{(k)}G_{il}^{(k)}\pbb{H_{lk} + \frac{f}{N}}\,.
\end{equation*}
The right-hand side can be estimated similarly to the right-hand side of \eqref{EqAuxiliere} to obtain \eqref{O4.66}; we omit the details.
\end{proof}

The following lemma is a standard stability estimate of the self-consistent equation associated with the semicircle law (see e.g.\ \cite[Lemma 5.5]{BK16}). Here we give a version with a sharp constant.
\begin{lemma}\label{lemma:inequality_s}
Let $z=E+\ii \eta \in \bf S$. Let $m,\tilde{m}$ be the solutions of the equation 
\[
x^2+zx+1=0\,.
\]
If $s$ satisfies $s^2+zs+1=r$ then
\begin{equation} \label{O414}
|s-m|\wedge|s-\tilde{m}| \leq  \sqrt{\abs{r}}\,.
\end{equation}
\end{lemma}
\begin{proof}
We choose the branch cut on the positive real axis, which gives $\im \sqrt{z} \geq 0$ for all $z \in \bb C \setminus \R_+$.
Thus 
\begin{equation*}
m=\frac{-z+\sqrt{z^2-4}}{2} \quad \mbox{and} \quad \tilde{m}=\frac{-z-\sqrt{z^2-4}}{2}\,.
\end{equation*}
Also, we have
\[
s=\frac{-z+\sqrt{z^2-4+4r}}{2}\,.
\]
Then \eqref{O414} follows from the fact that
\[
|\sqrt{a+b}-\sqrt{a}| \wedge |\sqrt{a+b}+\sqrt{a}| \leq \frac{|b|}{\sqrt{|a|+|b|}} \leq \sqrt{\abs{b}}
\]
for any $a,b \in \bb C$ and any complex square root $\sqrt{\cdot}$.
\end{proof}

Next, we set up the bootstrapping argument. Fix $E \in \bb R$. Let $K \deq \max \{k \in \bb N\col 1-kN^{-2}\geq N^{-1}\}$, and for $k =0,1,\dots,K$, we define the spectral parameter $z_k\deq E+\ii \eta_k$. We abbreviate
\begin{equation}
\zeta_k\deq \Big(\frac{r}{q^2}\Big)^{1/4}+\frac{r}{(N\eta_k)^{1/6}}+\frac{r}{(\log N+\log \eta_k) q }+\frac{f}{(N\eta_k)^{1/4}}
\end{equation}
where $\eta_k \deq 1-kN^{-2}$. 

The main work is to prove the following result, which states that Theorem \ref{maintheorem} holds for the spectral parameter $z$ being in the lattice $\{z_0,z_1,\dots, z_K\}$. To simplify notation, we use the variable $\xi\deq t/750$ instead of the $t$ from Theorem \ref{maintheorem}.

\begin{proposition} \label{prop3.5}
Let $\xi > 0$. Suppose that $2\leq r\leq q^2\leq N$, and $750 \, \xi \zeta_k \leq 1$ for all $k=0,1,\dots,K$. We have
\begin{equation}  \label{Oresult}
\bb P\Big(\max_{i,j}|G_{ij}(z_k)-m(z_k)\delta_{ij}|>360\xi\zeta_k \Big)\leq (3k+1)N^3\xi^{-r}\,.
\end{equation}
\end{proposition}

\begin{proof}
Throughout the proof we assume $\xi \geq 1$, for otherwise the proof is trivial. We proceed by induction on $k$. Since the sequence $\zeta_k$ is increasing, we may without loss of generality assume that the condition $1200 \, \xi \zeta_k \leq 1$ holds for all $k = 0, \dots, K$. Since the spectral parameter $z$ varies in the proof, in this proof we always indicate it explicitly.

We define the events
\[
\Omega_k\deq \big\{|s(z_k)-m(z_k)|\leq 50\xi\zeta_k \big\}\,, \quad \Xi_k \deq \{\Gamma(z_k)\leq 3/2\}
\]
for $k=0,1,\dots,K$.

For $k=0$, we see that $\Gamma(z_0) \leq 1$ and hence $\phi(z_0)=1$. Thus $\bb P(\Xi_0)=1$. From \eqref{O4.1} we find
\begin{equation*}
1 + z s + s^2 = \frac{1}{N} \sum_i G_{ii} Y_i\,,
\end{equation*}
so that, by Minkowski's inequality and \eqref{O4.3},
	\begin{multline*}
		\|1+z_0 s(z_0)+s(z_0)^2\|_r=\|\phi(1+z_0 s(z_0)+s(z_0)^2)\|_r \leq \max_i\| \phi(z_0) Y_i(z_0) G_{ii}(z_0)\|_r\\ \leq  48^2\Big[\Big(\frac{r}{q^2}\Big)^{1/2}+\frac{r^2}{(N\eta_0)^{1/3}}+\Big(\frac{r}{(\log N+\log \eta_0)q}\Big)^2+\frac{f^2}{\sqrt{N\eta_0}}\Big] \leq  (50\zeta_0)^2\,.
	\end{multline*}
    Thus Chebyshev's inequality implies
	\begin{equation} \label{Otemp1}
	\bb P\Big(|1+z_0s(z_0)+s(z_0)^2|>(50\xi\zeta_0)^2\Big) \leq \xi^{-2r} 
	\end{equation}
Hence Lemma \ref{lemma:inequality_s} yields
	\begin{equation} \label{O4199}
	\bb P \Big(|s(z_0)-m(z_0)| \wedge |s(z_0)-\tilde{m}(z_0)| >  50\xi\zeta_0\Big)  \leq  \xi^{-2r}\,.
	\end{equation}
	Note that $\im s>0$, and for $\eta_0=1$, we easily see that $\im \tilde m(z_0) < -1$. Thus $|s(z_0)-\tilde{m}(z_0)|>1$, and together with $50\xi \zeta_0\leq1$ we get $\bb P(\Omega_0^c)\leq \xi^{-2r}$.
Similarly, by \eqref{O4.3} and $50\xi\zeta_0\leq 1$ we have
	\begin{equation} \label{O419}
	\max_i\bb P \Big(|Y_i(z_0)G_{ii}(z_0)| > 50\xi\zeta_0\Big) \leq \max_i\bb P \Big(|Y_i (z_0) G_{ii}(z_0)| > (50\xi\zeta_0)^2\Big) \leq \xi^{-2r}\,.
	\end{equation}
From \eqref{O4.1} we get
	\begin{equation*}
	1+(z+m)G_{ii}=(m-s)G_{ii}+Y_iG_{ii}\,,
	\end{equation*}
so that by $m^2+zm+1=0$ and the elementary estimate $|m|\leq1$ we have
	\begin{equation} \label{O421}
	|G_{ii}-m|\leq |(s-m)G_{ii}|+ \abs{Y_iG_{ii}}\,.
	\end{equation}
	By $\bb P(\Omega_0^c)\leq \xi^{-2r}$ and \eqref{O419}--\eqref{O421} we have
	\begin{equation} \label{O423}
	\max_i\bb P(|G_{ii}(z_0)-m(z_0)|>150\xi\zeta_0) \leq 2\xi^{-2r}\,.
	\end{equation} 
	Similarly, by \eqref{O4.4} we have
   \begin{equation}\label{O4.23}
   \max_{i\ne j}\bb P (|G_{ij}(z_0)|>12\xi\zeta_0) \leq \xi^{-r}\,.
   \end{equation}
   Thus a union bound shows
   \begin{equation}
   \bb P\Big(\max_{i,j}|G_{ij}(z_0)-m(z_0)\delta_{ij}|>150\xi\zeta_0 \Big)\leq 2N\xi^{-2r}+N^2\xi^{-r}\leq 2N^2\xi^{-r}\,,
   \end{equation}
   where in the last step we used without loss of generality $\xi\geq1$ (for otherwise the probability bound is trivial). This proves the case $k=0$.
   
  For $k \geq 1$, we introduce the threshold index $\tilde{K}\deq\min \{k\col k\leq K, 105\xi\zeta_k\geq |m(z_k)-\tilde{m}(z_k)|\}$. Note that 
   \begin{equation}
   |m-\tilde{m}|^4=\eta^4+2E^2\eta^2+8\eta^2+E^4+16-8E^2
   \end{equation}
   is an increasing function of $\eta$. Thus, the sequence $|m(z_k)-\tilde{m}(z_k)|$ is decreasing and the sequence $\zeta_k$ is increasing. Hence
   $105\xi\zeta_k< |m(z_k)-\tilde{m}(z_k)|$ for all $k<\tilde{K}$.
   
   \textit{Case 1:} $1\leq k<\tilde{K}$. Since $ |\dd\Gamma/ \dd \eta|\leq 1/\eta^2$, we have $\phi(z_k)=1$ on $\Xi_{k-1}$. As in \eqref{O4199}, we have
   \begin{equation} 
   	\bb P \bigg(\phi|s(z_k)-m(z_k)| \wedge \phi|s(z_k)-\tilde{m}(z_k)| >50 \, \xi\zeta_k\bigg)  \leq  \xi^{-2r}\,.
   	\end{equation}
   	By Lipschitz continuity, $N\eta\geq 1$, and $\xi \geq 1$, we have
    \[
    |s(z_k)-m(z_k)|\leq 50\xi\zeta_{k-1}+\frac{2}{N^2\eta^2}\leq  52 \xi \zeta_k
    \]
    on $\Omega_{k-1}$. Note that for $k<K$ we have
    \[
    |m(z_k)-\tilde{m}(z_k)|\geq 105 \, \xi\zeta_k\,,
    \]
   thus on $\Omega_{k-1}$ we have $|s(z_k)-m(z_k)| \wedge |s(z_k)-\tilde{m}(z_k)|=|s(z_k)-m(z_k)|$. We have therefore proved that
   \begin{equation} \label{O1}
   \bb P (\Omega_{k-1}\cap\Xi_{k-1}\cap \Omega_k^c) \leq \xi^{-2r}\,.
   \end{equation}
   As in \eqref{O423}, we can use \eqref{O4.3} and \eqref{O421} to show that
   \begin{equation} \label{O428}
   \max_i\bb P \pB{\Omega_{k-1}\cap\Xi_{k-1}\cap \{|G_{ii}(z_k)-m(z_k)|>150 \, \xi\zeta_k\}} \leq 2\xi^{-2r}\,.
   \end{equation}
   Also, by \eqref{O4.4} and \eqref{O4.66} we see that
   \begin{equation} \label{O429}
   \max_{i\ne j}\bb P (\Xi_{k-1}\cap \{|G_{ij}(z_k)|>12\xi\zeta_k\}) \leq \xi^{-r}
   \end{equation}
   and
   \begin{equation}
   \max_{i,j\ne l}\bb P (\Xi_{k-1}\cap \{|G_{ij}(z_k)-G_{ij}^{(l)}|>12\xi\zeta_k\}) \leq \xi^{-r}\,.
   \end{equation}
   Thus by $|m|\leq1 $ and a union bound we have
   \begin{equation}
   \bb P (\Omega_{k-1}\cap\Xi_{k-1}\cap \{\Gamma(z_k)>1+162 \, \xi\zeta_k\}) \leq 2N\xi^{-2r}+(N^2+N^3)\xi^{-r}\leq 2N^3\xi^{-r}.
   \end{equation}
Note that the assumption $750\xi\zeta_k\leq1 $ ensures $162 \, \xi \zeta_k<1/2$, we find
   \begin{equation} \label{O2}
   \bb P (\Omega_{k-1}\cap\Xi_{k-1}\cap \Xi_k^c) \leq 2N^3\xi^{-r}.
   \end{equation}
   Note that $\bb P(\Omega_0\cap \Xi_0)\geq 1-\xi^{-2r}
   \geq 1-\xi^{-r}$, thus \eqref{O1}, \eqref{O2}, and an induction argument shows
   \begin{equation} \label{O433}
   \bb P (\Omega_{k}\cap\Xi_{k}) \geq 1-(2kN^3+1)\xi^{-r}
   \end{equation}
   for all $k<\tilde{K}$. Thus \eqref{Oresult} for $1 \leq k < \tilde K$ follows from \eqref{O428}, \eqref{O429}, and a union bound.
   
   \textit{Case 2:} $\tilde{K}\leq k\leq K$. As in Case 1, we have  $\phi(z_k)=1$ on $\Xi_{k-1}$, and
   \begin{equation} 
   \bb P \bigg(\phi|s(z_k)-m(z_k)| \wedge \phi|s(z_k)-\tilde{m}(z_k)| > 50 \, \xi\zeta_k\bigg)  \leq  \xi^{-2r}\,.
   \end{equation}
   Note that $ |m(z_k)-\tilde{m}(z_k)|\leq 105\xi\zeta_k$, so that the triangle inequality yields
   \begin{equation}
   \P (\Xi_{k-1}\cap \{|s(z_k)-m(z_k)|>155\xi\zeta_k\}) \leq \xi^{-2r}\,.
   \end{equation}
   As in \eqref{O428}, \eqref{O429}, and \eqref{O2}, we can show that
   \begin{equation}  \label{O436}
   \max_i\bb P (\Xi_{k-1}\cap \{|G_{ii}(z_k)-m(z_k)|>360\xi\zeta_k\}) \leq 2\xi^{-2r}\,,
   \end{equation}
   \begin{equation}  \label{O437}
   \max_{i\ne j}\bb P (\Xi_{k-1}\cap \{|G_{ij}(z_k)|>12\xi\zeta_k\}) \leq \xi^{-r}\,,
   \end{equation}
   and
   \begin{equation}  \label{O438}
   \bb P (\Xi_{k-1}\cap \Xi_k^c) \leq 2N^3\xi^{-r},
   \end{equation}
   where in showing \eqref{O438} we used the fact that $750\xi\zeta_k \leq 1$. Note that \eqref{O433} shows $\bb P(\Xi_{\tilde{K}-1})\geq 1-(2(\tilde {K}-1)N^3+1)\xi^{-r}$, so that together with \eqref{O438} we have
   \begin{equation} 
   \bb P (\Xi_{k}) \geq 1-(2kN^3+1)\xi^{-r}
   \end{equation}
   for all $\tilde{K}\leq k \leq K$. Thus \eqref{Oresult} for $\tilde{K} \leq k \leq K$ follows from \eqref{O436}, \eqref{O437}, and a union bound.
\end{proof}

We may now easily conclude the proof of Theorem \ref{maintheorem}. Without loss of generality, we can assume that $r \leq q^2$, for otherwise the condition \eqref{t_zeta} implies that $t \leq 1$, in which case Theorem \ref{maintheorem} is trivially true. By using $K\leq N^2-N$, $\xi=t/750$, and $r\geq 10$, we deduce from \eqref{Oresult} that for $k=0,1,\dots,K$, we have
\begin{equation}
\bb P\Big(\max_{i,j}|G_{ij}(z_k)-m(z_k)\delta_{ij}|>t\zeta_k/2 \Big)\leq N^5\Big(\frac{1000}{t}\Big)^r\,.
\end{equation}
Note that $|\dd G/\dd \eta|,|\dd m/\dd \eta|\leq 1/\eta^2$, and $\zeta $ is a decreasing function of $\eta$. Thus for any $z \in \f S$, we have
\begin{equation*}
\bb P\Big(\max_{i,j}|G_{ij}-m\delta_{ij}|>\frac{t\zeta}{2}+\frac{2}{N^2\eta^2} \Big)\leq N^5\Big(\frac{1000}{t}\Big)^r\,.
\end{equation*}
Note that the above is trivial for $t\leq 1000$, thus by $1/(N\eta)^2\leq \zeta$ we get
\begin{equation} \label{O4.44}
\bb P\Big(\max_{i,j}|G_{ij}-m\delta_{ij}|>t\zeta \Big)\leq N^5\Big(\frac{1000}{t}\Big)^r\,.
\end{equation}
This proves Theorem \ref{maintheorem} for $C_*=1000$.

\appendix
\section{Proof of \eqref{semi}} \label{sec:appendix}

In \eqref{semiH}, let $r=\log N$ and $t=C_*\ee^{5+D}$ for some $D>0$, which yields
\[
\bb P \qB{\max_{i,j}|G_{ij}-m\delta_{ij}|\leq C_*\ee^{5+D}\zeta} \geq  1-N^{-D}\,.
\] 
Let us work under the assumption $\eta\geq N^{-1+\tau}$ for some $\tau>0$. By the definition of $\zeta$ in \eqref{zeta}, in order to show $C_*\ee^{5+D}\zeta<\delta$ for some given $\delta>0$, it suffices to show
\begin{equation} \label{A}
C_*\ee^{5+D}\max\bigg\{\pbb{\frac{\log N}{q^2}}^{1/4},\ \frac{\log N}{(N\eta)^{1/6}}\,,\ \frac{\log N}{(\log \eta+\log N)q}\,,\ \frac{f}{(N\eta)^{1/4}}\bigg\} <\delta/4\,.
\end{equation}
This leads to the three conditions
\begin{equation} \label{condition}
q \geq \max\Big\{\Big(\frac{4C_*\ee ^{5+D}}{\delta}\Big)^2\sqrt{\log N},\,\frac{4C_*\ee^{5+D}}{\tau\delta} \Big\}\,, \quad f\leq \frac{\delta}{4C_*\ee^{5+D}} \cdot N^{\tau/4}\,, \quad  N^{\tau}\geq \Big(\frac{4C_*\ee^{5+D}\log N}{\delta}\Big)^6\,.
\end{equation}
where the first condition comes from the first and third terms in \eqref{A}, the second condition comes from the last term in \eqref{A}, and the last condition comes from the second term in \eqref{A}.
 
 \begin{itemize}
 \item
 In order to prove \eqref{semi} for $q \geq C \sqrt{\log N}$, we shall assume
\begin{equation} \label{tau}
\tau \geq (\log N)^{-1/2}
\end{equation}
for the rest of this appendix. (One easily checks that for a much smaller choice of $\tau \ll (\log N)^{-1/2}$, the first condition in \eqref{condition} implies $q \gg \sqrt{\log N}$, i.e.\ $q$ has to be much larger than the critical scale.)

\item
By the first condition in \eqref{condition}, we need
\begin{equation} \label{C}
q \geq C \sqrt{\log N}\,, \quad \mbox{where}\quad C\equiv C(\delta,D)\deq\Big(\frac{4C_*\ee ^{5+D}}{\delta}\Big)^2\,.
\end{equation}

\item
By the second condition in \eqref{condition} and $f \leq q$, it suffices to have
\begin{equation} \label{qqq}
q \leq (\log N)^{10} 
\end{equation}
and
\[
 (\log N)^{10} \leq \frac{\delta}{4C_*\ee^{5+D}} \cdot N^{(\log N)^{-1/2}/4}\,,
\]
where the latter can be simplified to
\begin{equation} \label{N_0}
  N\geq \exp \qbb{\pbb{\ee^{10}\log \pbb{\frac{4 C_* \ee^{5 + D}}{\delta}}}^2} \eqd N_0(\delta, D)\,.
\end{equation}

\item
One easily checks that $N \geq N_0$ satisfies the last condition in \eqref{condition}.

Thus we see that \eqref{semi} is true provided \eqref{tau}--\eqref{N_0} holds. 

\end{itemize}

\subsection*{Acknowledgements}
The authors are supported by the Swiss National Science Foundation and the European Research Council.

{\small
	
	\bibliography{bibliography} 
	
	\bibliographystyle{amsplain}
}
\end{document}